\documentclass[12pt]{amsart}

\newtheorem{theorem}{Theorem}[section]
\newtheorem{definition}[theorem]{Definition}
\newtheorem{proposition}[theorem]{Proposition}

\begin{document}

\title{Deformed Fourier models with formal parameters}

\author{Teodor Banica}
\address{Cergy-Pontoise University, 95000 Cergy-Pontoise, France. {\tt teodor.banica@u-cergy.fr}}

\subjclass[2010]{16T05 (60B15)}
\keywords{Quantum permutation, Matrix model}

\begin{abstract}
The deformed Fourier matrices $H=F_M\otimes_QF_N$, with $Q\in\mathbb T^{MN}$, produce a matrix model $C(S_{MN}^+)\to M_{MN}(C(\mathbb T^{MN}))$. When $Q\in\mathbb T^{MN}$ is generic, the corresponding fiber can be investigated via algebraic techniques, and the main character law is asymptotically free Poisson. We present here an alternative point of view on these questions, using formal parameters instead of generic parameters, and analytic tools.
\end{abstract}

\maketitle

\section*{Introduction}

It is well-known that the unitary representations of a discrete group $\Gamma$ are in one-to-one correspondence with the representations of the group algebra $C^*(\Gamma)$. Now given a discrete subgroup $\Gamma\subset U_N$, we obtain a representation $\pi:C^*(\Gamma)\to M_N(\mathbb C)$. This representation is in general not faithul, its target algebra being finite dimensional. On the other hand, this representation ``reminds'' $\Gamma$. We say that $\pi$ is inner faithful.

The inner faithful representations can be in fact axiomatized in the general discrete quantum group context. Given such a quantum group $\Gamma$, and a representation $\pi:C^*(\Gamma)\to B$, one can construct a biggest quotient $\Gamma\to\Lambda$ producing a factorization $\pi:C^*(\Gamma)\to C^*(\Lambda)\to B$, and $\pi$ is called inner faithful when $\Gamma=\Lambda$. See \cite{bb1}.

This construction is of particular interest when formulated from a dual viewpoint, with $\Gamma=\widehat{G}$, and with $B=M_K(C(X))$ being a random matrix algebra. To be more precise, given a compact quantum group $G$, and a matrix model $\pi:C(G)\to M_K(C(X))$, one can construct a biggest closed subgroup $H\subset G$ producing a factorization $\pi:C(G)\to C(H)\to M_K(C(X))$, and $\pi$ is called inner faithful when $G=H$. See \cite{bb1}.

Generally speaking, an inner faithful model $\pi:C(G)\to M_K(C(X))$ can be regarded as being a source of interesting information about $G$, of both algebraic and analytic nature. Thus, we have here a new method for investigating the compact quantum groups. This method is alternative to the pure algebraic geometric point of view (``easiness'').

A number of tools for dealing with the inner faithful models have been developed, some of them being algebraic \cite{bb1}, \cite{bb2}, \cite{bic}, \cite{chi}, and some other, analytic \cite{bfs}, \cite{bcv}, \cite{sso}, \cite{wa2}. However, at the level of concrete examples, only a few models have been succesfully investigated, so far. Among them is the model $C(S_{MN}^+)\to M_{MN}(\mathbb C)$ coming from a deformed Fourier matrix $H=F_M\otimes_QF_N$, with parameter $Q\in\mathbb T^{MN}$.

The story with these latter models is long and twisted, and involved many people. As a brief summary, the development of the subject was as follows:
\begin{enumerate}
\item Given an arbitrary inner faithful model $C(G)\to M_K(\mathbb C)$, an abstract formula for the Haar integration over $G$, based on \cite{fsk}, was found in \cite{bfs}.

\item The representations $C(S_{MN}^+)\to M_{MN}(\mathbb C)$ coming from deformed Fourier matrices with generic parameters were studied in \cite{bb2}, using algebraic techniques.

\item Some applications of the integration formula in \cite{bfs}, to the deformed Fourier matrix representations, were found short afterwards, in \cite{ba2}.

\item In the meantime, the algebraic methods in \cite{bb2} were substantially extended, as to cover certain non-generic parameters $Q\in\mathbb T^{MN}$, in \cite{bic}.

\item In the meantime as well, a generalization of the integration formula in \cite{bfs}, covering the models $C(G)\to M_K(C(X))$, was found in  \cite{wa2}.

\item The integration formula in \cite{wa2} was applied to certain related representations, of type $C(S_{N^2}^+)\to M_{N^2}(C(U_N))$, in the recent paper \cite{bne}.
\end{enumerate}

The purpose of this paper is to study the deformed Fourier models, using analytic techniques. We will take advantage of the recent formula in \cite{wa2}, and investigate the full parametric model $\pi:C(S_{MN}^+)\to M_{MN}(C(\mathbb T^{MN}))$, instead of its individual fibers. The formula in \cite{wa2} will turn to apply well, and to lead to concrete results. As in \cite{bb2}, our main result will state that main character becomes free Poisson, in the $M=tN\to\infty$ limit. We will discuss as well a number of further properties of the main character.

These results can be deduced as well from \cite{bb2}, since in the probabilistic picture for the moments, the non-generic parameters do not count. However, we believe that having a fully analytic proof is a good thing. In short, following \cite{bne}, we have now a second concrete application of the integration formula in \cite{bfs}, \cite{wa2}. Our hope is that this formula can be applied to some other situations, and could eventually become a serious alternative to the Weingarten formula \cite{bco}, \cite{csn}, and to the ``easiness'' methods in general \cite{bsp}, \cite{rwe}.

The paper is organized as follows: 1-2 are preliminary sections, in 3-4 we study the truncated moments of the main character, in 5-6 we compute the plain moments of the main character, in 7-8 we work out a number of moment estimates, and in 9-10 we state and prove our main results, and we end with a few concluding remarks.

\section{Quantum groups}

We use the quantum group formalism of Woronowicz \cite{wo1}, \cite{wo2}, with the extra axiom $S^2=id$. That is, we consider pairs $(A,u)$ consisting of a $C^*$-algebra $A$, and a unitary matrix $u\in M_N(A)$, such that the following formulae define morphisms of $C^*$-algebras:
$$\Delta(u_{ij})=\sum_ku_{ik}\otimes u_{kj}\quad,\quad\varepsilon(u_{ij})=\delta_{ij}\quad,\quad S(u_{ij})=u_{ji}^*$$

These morphisms are called comultiplication, counit and antipode. The abstract spectum $G=Spec(A)$ is called compact quantum group, and we write $A=C(G)$. 

The example that we are interested in, due to Wang \cite{wa1}, is as follows:

\begin{definition}
$C(S_N^+)$ is the universal $C^*$-algebra generated by the entries of a $N\times N$ matrix $u=(u_{ij})$ which is magic, in the sense that its entries are projections ($p=p^*=p^2$), summing up to $1$ on each row and each column of $u$.
\end{definition}

This algebra satisfies Woronowicz's axioms, and the underlying noncommutative space $S_N^+$ is therefore a quantum group, called quantum permutation group. We have an inclusion $S_N\subset S_N^+$, which is an isomorphism at $N=1,2,3$, but not at $N\geq4$. See \cite{wa1}.

Now back to the general case, we have the following key notion, fom \cite{bb1}:

\begin{definition}
Let $\pi:C(G)\to M_K(C(T))$ be a $C^*$-algebra representation. 
\begin{enumerate}
\item The Hopf image of $\pi$ is the smallest quotient Hopf $C^*$-algebra $C(G)\to C(H)$ producing a factorization of type $\pi:C(G)\to C(H)\to M_K(C(T))$.

\item When the inclusion $H\subset G$ is an isomorphism, i.e. when there is no non-trivial factorization as above, we say that $\pi$ is inner faithful.
\end{enumerate}
\end{definition}

As a basic example, when $G=\widehat{\Gamma}$ is a group dual, $\pi$ must come from a group representation $\Gamma\to C(T,U_K)$, and the factorization in (1) is the one obtained by taking the image, $\Gamma\to\Gamma'\subset C(T,U_K)$. Thus $\pi$ is inner faithful when $\Gamma\subset C(T,U_K)$.

Also, given a compact group $G$, and elements $g_1,\ldots,g_K\in G$, we can consider the representation $\pi=\oplus_iev_{g_i}:C(G)\to\mathbb C^K$. The minimal factorization of $\pi$ is then via $C(G')$, with $G'=\overline{<g_1,\ldots,g_K>}$. Thus $\pi$ is inner faithful when $G=\overline{<g_1,\ldots,g_K>}$.

We recall that an Hadamard matrix is a square matrix $H\in M_N(\mathbb C)$ whose entries are on the unit circle, and whose rows are pairwise orthogonal. Given a parametric family of such matrices, $\{H^x|x\in T\}$, we can consider the corresponding element $H\in M_N(C(T))$, that we call as well Hadamard matrix. The relation with $S_N^+$ comes from:

\begin{definition}
Associated to $H\in M_N(C(T))$ Hadamard is the representation
$$\pi:C(S_N^+)\to M_N(C(T))\quad,\quad \pi(u_{ij}):x\to Proj(H_i^x/H_j^x)$$
where $H_1^x,\ldots,H_N^x\in\mathbb T^N$ are the rows of $H^x$, and the quotients are taken inside $\mathbb T^N$.
\end{definition}

Here the fact that the projections on the right form a magic matrix, and hence produce a representation of $C(S_N^+)$, follows from the Hadamard matrix condition.

The problem is that of computing the Hopf image of the above representation. There is only one basic example here, namely the one coming from the Fourier coupling $F_G\in M_{G\times\widehat{G}}(\mathbb C)$ of a finite abelian group $G$. Here the representation constructed above factorizes as $\pi:C(S_G^+)\to C(S_G)\to C(G)\to M_N(\mathbb C)$, and the Hopf image is $C(G)$.

In order to approach the problem, we use tools from \cite{bfs}, \cite{wa2}. Let us first go back to the general context of Definition 1.2, and assume that $T$ is a measured space, so that we have a trace $tr:M_K(C(T))\to\mathbb C$, given by $tr(M)=\frac{1}{K}\sum_{i=1}^K\int_XM_{ii}(x)dx$. 

We have then the following key result, from \cite{bfs}, \cite{wa2}:

\begin{proposition}
Given an inner faithful model $\pi:C(G)\to M_K(C(T))$, we have
$$\int_G=\lim_{k\to\infty}\frac{1}{k}\sum_{r=1}^k(tr\circ\pi)^{*r}$$
in moments, with the convolutions at right being given by $\phi*\psi=(\phi\otimes\psi)\Delta$.
\end{proposition}

\begin{proof}
This was proved in \cite{bfs} in the case $X=\{.\}$, using theory from \cite{fsk}, the idea being that the Haar state can be obtained by starting with an arbitrary positive linear functional, and then convolving. The general case was established in \cite{wa2}.
\end{proof}

In the case where $G$ has a fundamental corepresentation $u=(u_{ij})$, the above result has a more concrete formulation, of linear algebra flavor, as follows:

\begin{proposition}
Given an inner faithful model $\pi:C(G)\to M_K(C(T))$, mapping $u_{ij}\to U_{ij}$, the moments of $\chi=\sum_iu_{ii}$ with respect to $\int_G^r=(tr\otimes\pi)^{*r}$ are the numbers
$$c_p^r=Tr(T_p^r)\quad:\quad (T_p)_{i_1\ldots i_p,j_1\ldots j_p}=tr(U_{i_1j_1}\ldots U_{i_pj_p})$$
and these numbers converge with $r\to\infty$ to the moments of $\chi$ with respect to $\int_G$.
\end{proposition}

\begin{proof}
By evaluating $\int_G^r=(tr\otimes\pi)^{*r}$ on a product of coefficients, we obtain:
$$\int_G^ru_{i_1j_1}\ldots u_{i_pj_p}=(T_p^r)_{i_1\ldots i_p,j_1\ldots j_p}$$

Now by summing over $i_x=j_x$, this gives the formula in the statement.  See \cite{bfs}.
\end{proof}

We can apply Proposition 1.5 to the Hadamard representations, and we obtain:

\begin{theorem}
For the representation coming from $H\in M_N(C(T))$ we have
$$c_p^r=\frac{1}{N^{(p+1)r}}\int_{T^r}\sum_{i_1^1\ldots i_p^r}\sum_{j_1^1\ldots j_p^r}
\frac{H_{i_1^1j_1^1}^{x_1}H_{i_1^2j_2^1}^{x_1}}{H_{i_1^1j_2^1}^{x_1}H_{i_1^2j_1^1}^{x_1}}\ldots\frac{H_{i_p^1j_p^1}^{x_1}H_{i_p^2j_1^1}^{x_1}}{H_{i_p^1j_1^1}^{x_1}H_{i_p^2j_p^1}^{x_1}}\ldots\ldots\frac{H_{i_1^rj_1^r}^{x_r}H_{i_1^1j_2^r}^{x_r}}{H_{i_1^rj_2^r}^{x_r}H_{i_1^1j_1^r}^{x_r}}\ldots\frac{H_{i_p^rj_p^r}^{x_r}H_{i_p^1j_1^r}^{x_r}}{H_{i_p^rj_1^r}^{x_r}H_{i_p^1j_p^r}^{x_r}}dx$$
and these numbers converge with $r\to\infty$ to the moments of $\chi$ with respect to $\int_G$.
\end{theorem}

\begin{proof}
We have indeed the following computation:
\begin{eqnarray*}
c_p^r
&=&\sum_{i_1^1\ldots i_p^r}(T_p)_{i_1^1\ldots i_p^1,i_1^2\ldots i_p^2}\ldots\ldots(T_p)_{i_1^r\ldots i_p^r,i_1^1\ldots i_p^1}\\
&=&\int_{T^r}\sum_{i_1^1\ldots i_p^r}tr(U_{i_1^1i_1^2}^{x_1}\ldots U_{i_p^1i_p^2}^{x_1})\ldots\ldots tr(U_{i_1^ri_1^1}^{x_r}\ldots U_{i_p^ri_p^1}^{x_r})dx\\
&=&\frac{1}{N^r}\int_{T^r}\sum_{i_1^1\ldots i_p^r}\sum_{j_1^1\ldots j_p^r}(U_{i_1^1i_1^2}^{x_1})_{j_1^1j_2^1}\ldots(U_{i_p^1i_p^2}^{x_1})_{j_p^1j_1^1}\ldots\ldots(U_{i_1^ri_1^1}^{x_r})_{j_1^rj_2^r}\ldots(U_{i_p^ri_p^1}^{x_r})_{j_p^rj_1^r}dx
\end{eqnarray*}

In terms of $H$, this gives the formula in the statement. See \cite{ba2}.
\end{proof}

\section{Fourier models}

As mentioned in section 1, the ``simplest'' matrix model is the one coming from the Fourier matrix $F_G\in M_{G\times\widehat{G}}(\mathbb C)$ of a finite abelian group $G$, where the associated quantum group is $G$ itself. Our purpose here will be that of investigating the ``next simplest'' models. These appear by deforming the Fourier matrices, or rather the tensor products of such matrices, $F_{G\times H}=F_G\otimes F_H$, via the following construction, due to Di\c t\u a \cite{dit}:

\begin{proposition}
The matrix $\mathcal F_{G\times H}\in M_{G\times H}(\mathbb T^{G\times H})$ given by
$$(\mathcal F_{G\times H})_{ia,jb}(Q)=Q_{ib}(F_G)_{ij}(F_H)_{ab}$$
is complex Hadamard, and its fiber at $Q=(1_{ib})$ is the Fourier matrix $F_{G\times H}$.
\end{proposition}

\begin{proof}
The fact that the rows of $F_G\otimes_QF_H=\mathcal F_{G\times H}(Q)$ are pairwise orthogonal follows from definitions, see \cite{dit}. With $1=(1_{ij})$ we have $(F_G\otimes_1F_H)_{ia,jb}=(F_G)_{ij}(F_H)_{ab}$, and we recognize here the formula of $F_{G\times H}=F_G\otimes F_H$, in double index notation.
\end{proof}

The fibers $F_G\otimes_QF_H=\mathcal F_{G\times H}(Q)$ were investigated in \cite{bb2}, and then in \cite{bic}, by using algebraic techniques. Our purpose here is that of obtaining some related results, regarding the matrix $\mathcal F_{G\times H}$ itself, by using analytic techniques. We have:

\begin{theorem}
For the representation coming from $\mathcal F_{G\times H}$ we have
$$c_p^r=\frac{1}{M^{r+1}N}\#\left\{\begin{matrix}i_1,\ldots,i_r,a_1,\ldots,a_p\in\{0,\ldots,M-1\},\\
b_1,\ldots,b_p\in\{0,\ldots,N-1\},\\
[(i_x+a_y,b_y),(i_{x+1}+a_y,b_{y+1})|y=1,\ldots,p]\\
=[(i_x+a_y,b_{y+1}),(i_{x+1}+a_y,b_y)|y=1,\ldots,p], \forall x
\end{matrix}\right\}$$
where $M=|G|,N=|H|$, and the sets between brackets are sets with repetitions.
\end{theorem}

\begin{proof}
We use the formula in Theorem 1.6. With $K=F_G$, $L=F_H$ we have:
\begin{eqnarray*}
c_p^r
&=&\frac{1}{(MN)^r}\int_{T^r}\sum_{i_1^1\ldots i_p^r}\sum_{b_1^1\ldots b_p^r}\frac{Q^1_{i_1^1b_1^1}Q^1_{i_1^2b_2^1}}{Q^1_{i_1^1b_2^1}Q^1_{i_1^2b_1^1}}\ldots\frac{Q^1_{i_p^1b_p^1}Q^1_{i_p^2b_1^1}}{Q^1_{i_p^1b_1^1}Q^1_{i_p^2b_p^1}}\ldots\ldots\frac{Q^r_{i_1^rb_1^r}Q^r_{i_1^1b_2^r}}{Q^r_{i_1^rb_2^r}Q^r_{i_1^1b_1^r}}\ldots\frac{Q^r_{i_p^rb_p^r}Q^r_{i_p^1b_1^r}}{Q^r_{i_p^rb_1^r}Q^r_{i_p^1b_p^r}}\\
&&\hskip15mm\frac{1}{M^{pr}}\sum_{j_1^1\ldots j_p^r}\frac{K_{i_1^1j_1^1}K_{i_1^2j_2^1}}{K_{i_1^1j_2^1}K_{i_1^2j_1^1}}\ldots\frac{K_{i_p^1j_p^1}K_{i_p^2j_1^1}}{K_{i_p^1j_1^1}K_{i_p^2j_p^1}}\ldots\ldots\frac{K_{i_1^rj_1^r}K_{i_1^1j_2^r}}{K_{i_1^rj_2^r}K_{i_1^1j_1^r}}\ldots\frac{K_{i_p^rj_p^r}K_{i_p^1j_1^r}}{K_{i_p^rj_1^r}K_{i_p^1j_p^r}}\\
&&\hskip15mm\frac{1}{N^{pr}}\sum_{a_1^1\ldots a_p^r}\frac{L_{a_1^1b_1^1}L_{a_1^2b_2^1}}{L_{a_1^1b_2^1}L_{a_1^2b_1^1}}\ldots\frac{L_{a_p^1b_p^1}L_{a_p^2b_1^1}}{L_{a_p^1b_1^1}L_{a_p^2b_p^1}}\ldots\ldots\frac{L_{a_1^rb_1^r}L_{a_1^1b_2^r}}{L_{a_1^rb_2^r}L_{a_1^1b_1^r}}\ldots\frac{L_{a_p^rb_p^r}L_{a_p^1b_1^r}}{L_{a_p^rb_1^r}L_{a_p^1b_p^r}}\,dQ\\
\end{eqnarray*}

Since we are in the Fourier matrix case, $K=F_G,L=F_H$, we can perform the sums over $j,a$. To be more precise, the last two averages appearing above are respectively:
\begin{eqnarray*}
\Delta(i)&=&\prod_x\prod_y\delta(i^x_y+i^{x+1}_{y-1},i^{x+1}_y+i^x_{y-1})\\
\Delta(b)&=&\prod_x\prod_y\delta(b^x_y+b^{x+1}_{y-1},b^{x+1}_y+b^x_{y-1})
\end{eqnarray*}

We therefore obtain the following formula for the truncated moments of the main character, where $\Delta$ is the product of Kronecker symbols constructed above:
$$c_p^r=\frac{1}{(MN)^r}\int_{T^r}\sum_{\Delta(i)=\Delta(b)=1}\frac{Q^1_{i_1^1b_1^1}Q^1_{i_1^2b_2^1}}{Q^1_{i_1^1b_2^1}Q^1_{i_1^2b_1^1}}\ldots\frac{Q^1_{i_p^1b_p^1}Q^1_{i_p^2b_1^1}}{Q^1_{i_p^1b_1^1}Q^1_{i_p^2b_p^1}}\ldots\ldots\frac{Q^r_{i_1^rb_1^r}Q^r_{i_1^1b_2^r}}{Q^r_{i_1^rb_2^r}Q^r_{i_1^1b_1^r}}\ldots\frac{Q^r_{i_p^rb_p^r}Q^r_{i_p^1b_1^r}}{Q^r_{i_p^rb_1^r}Q^r_{i_p^1b_p^r}}\,dQ$$

Now by integrating with respect to $Q\in(\mathbb T^{G\times H})^r$, we are led to counting the multi-indices $i,b$ satisfying the condition $\Delta(i)=\Delta(b)=1$, along with the following conditions, where the sets between brackets are by definition sets with repetitions:
$$\begin{bmatrix}
i_1^1b_1^1&\ldots&i_p^1b_p^1&i_1^2b_2^1&\ldots&i_p^2b_1^1
\end{bmatrix}=
\begin{bmatrix}
i_1^1b_2^1&\ldots&i_p^1b_1^1&i_1^2b_1^1&\ldots&i_p^2b_p^1
\end{bmatrix}$$
$$\vdots$$
$$\begin{bmatrix}
i_1^rb_1^r&\ldots&i_p^rb_p^r&i_1^1b_2^r&\ldots&i_p^1b_1^r
\end{bmatrix}
=\begin{bmatrix}
i_1^rb_2^r&\ldots&i_p^rb_1^r&i_1^1b_1^r&\ldots&i_p^1b_p^r
\end{bmatrix}$$

In a more compact notation, the moment formula is therefore as follows:
$$c_p^r=\frac{1}{(MN)^r}\#\left\{i,b\Big|\Delta(i)=\Delta(b)=1,\ [i^x_yb^x_y,i^{x+1}_yb^x_{y+1}]=[i^x_yb^x_{y+1},i^{x+1}_yb^x_y],\forall x\right\}$$

Now observe that the above Kronecker type conditions $\Delta(i)=\Delta(b)=1$ tell us that the arrays of indices $i=(i^x_y),b=(b^x_y)$ must be of the following special form:
$$\begin{pmatrix}i^1_1&\ldots&i^1_p\\&\ldots\\ i^1_r&\ldots&i^r_p\end{pmatrix}=\begin{pmatrix}i_1+a_1&\ldots&i_1+a_p\\&\ldots\\ i_r+a_1&\ldots&i_r+a_p\end{pmatrix}\ ,\ 
\begin{pmatrix}b^1_1&\ldots&b^1_p\\&\ldots\\ b^1_r&\ldots&b^r_p\end{pmatrix}=\begin{pmatrix}j_1+b_1&\ldots&j_1+b_p\\&\ldots\\ j_r+b_1&\ldots&j_r+b_p\end{pmatrix}$$

Here all the new indices $i_x,j_x,a_y,b_y$ are uniquely determined, up to a choice of $i_1,j_1$. Now by replacing $i^x_y,b^x_y$ with these new indices $i_x,j_x,a_y,b_y$, with a $MN$ factor added, which accounts for the choice of $i_1,j_1$, we obtain the following formula:
$$c_p^r=\frac{1}{(MN)^{r+1}}\#\left\{i,j,a,b\Big|\begin{matrix}[(i_x+a_y,j_x+b_y),(i_{x+1}+a_y,j_x+b_{y+1})]\\
=[(i_x+a_y,j_x+b_{y+1}),(i_{x+1}+a_y,j_x+b_y)],\forall x\end{matrix}\right\}$$

Now observe that we can delete if we want the $j_x$ indices, which are irrelevant. Thus, we obtain the formula in the statement.
\end{proof}

Summarizing, the Haar integration formula in \cite{wa2} leads to a combinatorial interpretation of the moments of the main character. In what follows we will investigate these moments, first with some exact computations, and then with analytic techniques.

\section{Exact computations}

In this section and in the next one we study the numbers $c_p^r$ found in Theorem 2.2, with a number of exact computations. Observe first that these numbers depend only on $M=|G|$ and $N=|H|$. We denote in what follows these numbers by $c_p^r(M,N)$.

As an illustration, here are a few trivial computations:

\begin{proposition}
The numbers $c_p^r(M,N)$ have the following properties:
\begin{enumerate}
\item $c_p^r(1,N)=N^{p-1}$.

\item $c_p^r(M,1)=M^{p-1}$.

\item $c_1^r(M,N)=1$.

\item $c_p^1(M,N)=(MN)^{p-1}$.
\end{enumerate}
\end{proposition}

\begin{proof}
In all the cases under investigation, the conditions on the sets with repetitions in Theorem 2.2 are trivially satisfied, and this gives the above formulae.
\end{proof}

We have in fact the following result, including all the ``obvious'' information:

\begin{proposition}
The following normalized quantities belong to $[0,1]$,
$$d_p^r(M,N)=\frac{1}{(MN)^{p-1}}\cdot c_p^r(M,N)$$
and are equal to $1$ at $M=1,N=1,p=1$ or $r=1$. 
\end{proposition}

\begin{proof}
According to Theorem 2.2, the rescaled moments are given by:
$$d_p^r(M,N)=\frac{1}{M^{p+r}N^p}\#\left\{\begin{matrix}i_1,\ldots,i_r,a_1,\ldots,a_p\in\{0,\ldots,M-1\},\\
b_1,\ldots,b_p\in\{0,\ldots,N-1\},\\
[(i_x+a_y,b_y),(i_{x+1}+a_y,b_{y+1})]\\
=[(i_x+a_y,b_{y+1}),(i_{x+1}+a_y,b_y)],\forall x
\end{matrix}\right\}$$

Thus $d_p^r(M,N)\in[0,1]$, and the other assertions follow from Proposition 2.1. 
\end{proof}

Let us perform now some computations. The formulae look better for the numbers $d_p^r(M,N)$ in Proposition 3.2, so we will use these numbers. First, we have:

\begin{proposition}
When one of $i,a,b$ consists of equal indices, the conditions defining $d_p^r(M,N)$ are trivially satisfied. The corresponding contribution is
$$\alpha_p^r(M,N)=1-\frac{(M^p-M)(M^r-M)(N^p-N)}{M^{p+r}N^p}$$
and this quantity equals $d_p^r(M,N)$ at $M=1$, $N=1$, $r=1$, or $p\leq2$.
\end{proposition}

\begin{proof}
Assume that one of $i,a,b$ consists of equal indices. By translation we can assume that this common index is $0$, and the conditions defining $d_p^r(M,N)$ read:
\begin{eqnarray*}
i_x=0&:&[(a_y,b_y),(a_y,b_{y+1})]=[(a_y,b_{y+1}),(a_y,b_y)]\\
a_y=0&:&[(i_x,b_y),(i_{x+1},b_{y+1})]=[(i_x,b_{y+1}),(i_{x+1},b_y)]\\
b_y=0&:&[(i_x+a_y,0),(i_{x+1}+a_y,0)]=[(i_x+a_y,0),(i_{x+1}+a_y,0)]
\end{eqnarray*}

Thus the conditions are trivially satisfied when $i_x=0$ or $b_y=0$, and the same happens when $a_y=0$, by performing a cyclic permutation on the $y$ indices.

The number of situations where one of $i,a,b$ consists of equal indices is:
$$K=M^{p+r}N^p-(M^p-M)(M^r-M)(N^p-N)$$

By dividing by $M^{p+r}N^p$, we obtain the formula in the statement. 

The assertions about $M=1,N=1,p=1,r=1$ are clear, because in all these cases the product in the definition of $\alpha_p^r(M,N)$ vanishes, and so $\alpha_p^r(M,N)=1$.

Finally, at $p=2$, the equations defining $d_2^r(M,N)$ are as follows:
\begin{eqnarray*}
&&[(i_x+a_1,b_1),(i_x+a_2,b_2),(i_{x+1}+a_1,b_2),(i_{x+1}+a_2,b_1)]\\
&=&[(i_x+a_1,b_2),(i_x+a_2,b_1),(i_{x+1}+a_1,b_1),(i_{x+1}+a_2,b_2)],\forall x
\end{eqnarray*}

We already know that these conditions are satisfied when $a_1=a_2$ or $b_1=b_2$. So, assume $a_1\neq a_2,b_1\neq b_2$. The element $(i_x+a_1,b_1)$ must appear somewhere at right, and the only possible choice is $(i_x+a_1,b_1)=(i_{x+1}+a_1,b_1)$, which gives $i_x=i_{x+1}$. Thus, all the $i_x$ indices must be are equal, and we are done.
\end{proof}

In general, the situation is more complicated. As a first remark, we have: 

\begin{proposition}
We have $d_p^r(M,N)\geq\delta_p(M,N)$, where
$$\delta_p(M,N)=\frac{1}{(MN)^p}\#\left\{\begin{matrix}a_1,\ldots,a_p\in\{0,\ldots,M-1\}\\ b_1,\ldots,b_p\in\{0,\ldots,N-1\}\end{matrix}\Big|\begin{matrix}[(a_1,b_1),(a_2,b_2),\ldots,(a_p,b_p)]\ \ \ \ \\=[(a_1,b_p),(a_2,b_1),\ldots,(a_p,b_{p-1})]\end{matrix}\right\}$$
where the sets between brackets are as usual sets with repetitions.
\end{proposition}

\begin{proof}
This is indeed clear from the fact that the conditions defining $\delta_p^r(M,N)$ are trivially satisfied when the indices $a,b$ satisfy $[(a_y,b_y)]=[(a_y,b_{y+1})]$.
\end{proof}

We can merge and extend Proposition 3.3 and Proposition 3.4, as follows:

\begin{theorem}
When $i$ consists of equal indices, or when $[(a_y,b_y)]=[(a_y,b_{y+1})]$, the conditions defining $d_p^r(M,N)$ are trivially satisfied. The corresponding contribution is
$$\beta_p^r(M,N)=\delta_p(M,N)+\frac{1}{M^{r-1}}(1-\delta_p(M,N))$$
and this quantity equals $d_p^r(M,N)$ at $M=1$, $N=1$, $r=1$, or $p\leq3$.
\end{theorem}

\begin{proof}
The first assertion is clear, and by definition of $
\delta_p(M,N)$, the corresponding contribution is the one in the statement. Since at $M=1$, $N=1$, $r=1$ or $p\leq2$ we have $\beta_p^r(M,N)=\alpha_p^r(M,N)$, the results here follow from Proposition 3.3. 

It remains to discuss the case $p=3$. Here the equations are as follows:
\begin{eqnarray*}
&&[(i_x+a_1,b_1),(i_x+a_2,b_2),(i_x+a_3,b_3),(i_{x+1}+a_1,b_2),(i_{x+1}+a_2,b_3),(i_{x+1}+a_3,b_1)]\\
&=&[(i_x+a_1,b_2),(i_x+a_2,b_3),(i_x+a_3,b_1),(i_{x+1}+a_1,b_1),(i_{x+1}+a_2,b_2),(i_{x+1}+a_3,b_3)]
\end{eqnarray*}

We must prove that all the solutions are trivial, in the sense that either all the $i_x$ are equal, or the following condition is satisfied:
$$[(a_1,b_1),(a_2,b_2),(a_3,b_3)]=[(a_1,b_2),(a_2,b_3),(a_3,b_1)]$$

So, assume that we are in the non-trivial case, and pick $x$ such that $i_x\neq i_{x+1}$. Let us look now at the first element appearing on the left in the above equation, namely $(i_1+a_1,b_1)$. Since this element must appear as well on the right, we have 6 cases to be investigated. Observe now that in these 6 cases we must have, respectively:
$$b_1=b_2,b_1=b_3,a_1=a_3,i_x=i_{x+1},b_1=b_2,b_1=b_3$$

Thus, we have one case which is impossible, namely the one needing $i_x=i_{x+1}$, and in the other 5 cases, we always obtain a relation of type $a_i=a_j$ or $b_i=b_j$, with $i\neq j$.

So, assume $a_i=a_j$, with $i\neq j$. By using a cyclic permutation of the indices, we can assume that we have $a_2=a_3$. Now observe that our equations simplify, as follows:
\begin{eqnarray*}
&&[(i_x+a_1,b_1),(i_x+a_2,b_2),\underline{(i_x+a_2,b_3)},(i_{x+1}+a_1,b_2),\underline{(i_{x+1}+a_2,b_3)},(i_{x+1}+a_2,b_1)]\\
&=&[(i_x+a_1,b_2),\underline{(i_x+a_2,b_3)},(i_x+a_2,b_1),(i_{x+1}+a_1,b_1),(i_{x+1}+a_2,b_2),\underline{(i_{x+1}+a_2,b_3)}]
\end{eqnarray*}

As for the condition $[(a_y,b_y)]\neq[(a_y,b_{y+1})]$, this simplifies as well, as follows:
$$[(a_1,b_1),(a_2,b_2),\underline{(a_2,b_3)}]\neq[(a_1,b_2),\underline{(a_2,b_3)},(a_2,b_1)]$$

Summarizing, the simplifications make dissapear the variables $a_3,b_3$, and so we are led to a $p=2$ problem, where the solutions are already known to be trivial. 

In the case $b_i=b_j$, with $i\neq j$, the situation is similar. By cyclic permutation we can assume $b_1=b_3$, and our equations simplify, as follows:
\begin{eqnarray*}
&&[(i_x+a_1,b_1),(i_x+a_2,b_2),\underline{(i_x+a_3,b_1)},(i_{x+1}+a_1,b_2),(i_{x+1}+a_2,b_1),\underline{(i_{x+1}+a_3,b_1)}]\\
&=&[(i_x+a_1,b_2),(i_x+a_2,b_1),\underline{(i_x+a_3,b_1)},(i_{x+1}+a_1,b_1),(i_{x+1}+a_2,b_2),\underline{(i_{x+1}+a_3,b_1)}]
\end{eqnarray*}

As for the condition $[(a_y,b_y)]\neq[(a_y,b_{y+1})]$, this simplifies as well, as follows:
$$[(a_1,b_1),(a_2,b_2),\underline{(a_3,b_1)}]\neq[(a_1,b_2),(a_2,b_1),\underline{(a_3,b_1)}]$$

Thus, we are led once again to a $p=2$ problem, whose solutions are trivial.
\end{proof}

\section{Higher truncations}

We know from Theorem 3.5 above that at small values of the truncation parameter, namely $p=1,2,3$, the numbers $d_p^r(M,N)$ come only from ``trivial contributions''.

At $p=4$ and higher the situation becomes considerably more complex, involving the arithmetics of $M,N$, and this even in the simplest case, $r=2$. 

We have here the following result, that we won't use in what follows, but which might be interesting for instance in connection with the speculations in \cite{ba1}:

\begin{theorem}
We have the formula
$$d_4^2(M,N)=\beta_4^2(M,N)+\delta_{2|M}\frac{(M-2)(N-1)}{M^4N^3}$$
where $\delta_{2|M}\in\{0,1\}$ is equal to $0$ when $M$ is odd, and to $1$ when $M$ is even.
\end{theorem}

\begin{proof}
We have two equations, the one at $x=1$ being as follows:
\begin{eqnarray*}
&&[(i_1+a_1,b_1),\ldots,(i_1+a_4,b_4),(i_2+a_1,b_2),\ldots,(i_2+a_4,b_1)]\\
&=&[(i_1+a_1,b_2),\ldots,(i_1+a_4,b_1),(i_2+a_1,b_1),\ldots,(i_2+a_4,b_4)]
\end{eqnarray*}

As for the equation at $x=2$, this is as follows:
\begin{eqnarray*}
&&[(i_2+a_1,b_1),\ldots,(i_2+a_4,b_4),(i_1+a_1,b_2),\ldots,(i_1+a_4,b_1)]\\
&=&[(i_2+a_1,b_2),\ldots,(i_2+a_4,b_1),(i_1+a_1,b_1),\ldots,(i_1+a_4,b_4)]
\end{eqnarray*}

Since these equations are equivalent, we are left with the $x=1$ equation.

In order to compute the non-trivial contributions, we can assume $i_1\neq i_2$. Let us look at the first element appearing on the left, $(i_1+a_1,b_1)$. Since this element must appear as well on the right, we have 8 cases to be investigated. In these 8 cases, we must have:
$$b_1=b_2,a_1=a_2,b_1=b_4,a_1=a_4,i_1=i_2,b_1=b_2,(i_1+a_1,b_1)=(i_2+a_3,b_3),b_1=b_4$$

Thus one case is impossible, 6 cases reduce to the case $p=3$, by using a cyclic reduction, as in the proof of Theorem 3.5, and there is one case left, $(i_1+a_1,b_1)=(i_2+a_3,b_3)$.

The same argument applies to the other 7 elements appearing on the left, and we conclude that the non-trivial solutions could only come from:
$$(i_1+a_x,b_x)=(i_2+a_{x+2},b_{x+2})\quad,\quad (i_2+a_x,b_{x+1})=(i_1+a_{x+2},b_{x+3})$$

Thus our indices $i,a,b$ must be of the following special form, with $2i=0$:
$$\begin{cases}
i&=(i_1,i+i_1)\\
a&=(a_1,a_2,i+a_1,i+a_2)\\
b&=(b_1,b_2,b_1,b_2)
\end{cases}$$

In order to find now the non-trivial solutions, we must assume that we have $i\neq0$, and $[(a_y,b_y)]\neq[(a_y,b_{y+1})]$. But, by translating by $i_1$, this latter condition reads:
$$[(a_1,b_1),(a_2,b_2),(i+a_1,b_1),(i+a_2,b_2)]\neq[(a_1,b_2),(a_2,b_1),(i+a_1,b_2),(i+a_2,b_1)]$$

Thus we must have $b_1\neq b_2$, and $a_1\neq a_2$, $a_1\neq i+a_2$ as well.

We can now compute the non-trivial contribution. This is given by:
$$K=\frac{1}{M^6N^4}\cdot M\delta_{2|M}\cdot M(M-2)\cdot N(N-1)$$

To be more precise, $\frac{1}{M^6N^4}$ is the normalization factor from the definition of $d_4^2(M,N)$, then $M\delta_{2|M}$ comes from the choice of $i_1$ and of $i\neq0$ satisfying $2i=0$, then $M(M-2)$ comes from the choice of $a_1$ and of $a_2\neq a_1,i+a_1$, and finally $N(N-1)$ comes from the choice of $b_1=b_2$. But this gives the formula in the statement, and we are done.
\end{proof}

As a conclusion, the exact computation of $d_p^r(M,N)$ is an interesting problem. In what follows we will only study the asymptotics of these numbers, with the result that the estimate $d_p^r(M,N)\geq\beta_p^r(M,N)$ from Theorem 3.5 becomes an equality, with $r\to\infty$.

\section{Limiting moments}

Let us go back now to the numbers $\delta_p(M,N)$, from Proposition 3.4 above. 

These numbers are known since \cite{bb2} to be the rescaled moments of the main character for the matrix model associated to $\mathcal F_{G\times H}(Q)$, where $|G|=M,|H|=N$, and where $Q\in\mathbb T^{G\times H}$ is generic. We will prove now that our moments are precisely these numbers:
$$\lim_{r\to\infty}d_p^r(M,N)=\delta_p(M,N)$$

For this purpose, observe that both $d_p^r(M,N),\delta_p(M,N)$ count, modulo some normalizations, the solutions of certain equations on the indices  $a_1,\ldots,a_p\in\{0,\ldots,M-1\}$ and $b_1,\ldots,b_p\in\{0,\ldots,N-1\}$. We will prove the convergence componentwise, with respect to these pairs of multi-indices $(a,b)$. We use the following simple fact:

\begin{proposition}
We have $[a_y]=[b_y]$ inside a finite abelian group $G$ precisely when
$$\sum_y\chi(a_y)=\sum_y\chi(b_y)$$
as an equality of complex numbers, for any character $\chi\in\widehat{G}$.
\end{proposition}

\begin{proof}
By linearity, we have the following equivalences:
\begin{eqnarray*}
[a_y]=[b_y]
&\iff&\sum_ya_y=\sum_yb_y\ {\rm inside}\ C^*(G)\\
&\iff&\varphi\left(\sum_ya_y\right)=\varphi\left(\sum_yb_y\right),\forall\varphi\in C(G)\\
&\iff&\chi\left(\sum_ya_y\right)=\chi\left(\sum_yb_y\right),\forall\chi\in\widehat{G}
\end{eqnarray*}

Thus, we obtain the condition in the statement.
\end{proof}

Now back to our question, since only the cardinalities $M=|G|,N=|H|$ are revelant, we can assume $G=\mathbb Z_M,H=\mathbb Z_N$. We first have the following technical result:

\begin{proposition}
For a pair of multi-indices $(a,b)$, the following are equivalent:
\begin{enumerate}
\item $[(a_y,b_y)]=[(a_y,b_{y+1})]$.

\item $[(i+a_y,b_y),(a_y,b_{y+1})]=[(i+a_y,b_{y+1}),(a_y,b_y)]$, for any $i\in\mathbb Z_M$.
\end{enumerate}
\end{proposition}

\begin{proof}
Observe that $(1)\implies(2)$ is clear. For $(2)\implies(1)$, we use Proposition 5.1. By using the identification $\widehat{\mathbb Z_M\times\mathbb Z_N}\simeq\widehat{\mathbb Z_M}\times\widehat{\mathbb Z_N}$, we have, with $\eta\in\widehat{\mathbb Z_M},\rho\in\widehat{\mathbb Z_N}$:
\begin{eqnarray*}
&&[(i+a_y,b_y),(a_y,b_{y+1})]=[(i+a_y,b_{y+1}),(a_y,b_y)],\forall i\\
&\iff&\sum_y\eta(i+a_y)\rho(b_y)+\eta(a_y)\rho(b_{y+1})=\sum_y\eta(i+a_y)\rho(b_{y+1})+\eta(a_y)\rho(b_y),\forall i,\eta,\rho\\
&\iff&\eta(i)\sum_y\eta(a_y)\rho(b_y)-\eta(a_y)\rho(b_{y+1})=\sum_y\eta(a_y)\rho(b_y)-\eta(a_y)\rho(b_{y+1}),\forall i,\eta,\rho\\
&\iff&\sum_y\eta(a_y)\rho(b_y)-\eta(a_y)\rho(b_{y+1})=0,\forall\eta,\rho
\iff[(a_y,b_y)]=[(a_y,b_{y+1})]
\end{eqnarray*}

Thus, we have obtained the equivalence in the statement.
\end{proof}

With the above result in hand, we can prove the estimate that we need, namely:

\begin{proposition}
Assuming $[(a_y,b_y)]\neq [(a_y,b_{y+1})]$, the number
$$K_p^r(a,b)=\frac{1}{M^r}\#\left\{i_1,\ldots,i_r\leq M\Big|\begin{matrix}[(i_x+a_y,b_y),(i_{x+1}+a_y,b_{y+1})]\ \ \ \ \\=[(i_x+a_y,b_{y+1}),(i_{x+1}+a_y,b_y)],\forall x\end{matrix}\right\}$$
goes to $0$ in the $r\to\infty$ limit.
\end{proposition}

\begin{proof}
Observe that the problem is already solved at $p\leq 3$, because by Theorem 3.5 all the $i_x$ indices must be equal, and so the number in the statement is:
$$K_2^r(a,b)=\frac{1}{M^{r-1}}\to0$$

In general now, consider the set $S\subset\{0,\ldots,M-1\}$ consisting of the solutions $i$ of the following equation:
$$[(i+a_y,b_y),(a_y,b_{y+1})]=[(i+a_y,b_{y+1}),(a_y,b_y)]$$

In terms of this set, the quantity in the statement is given by:
$$K_p^r(a,b)=\frac{1}{M^r}\#\left\{i_1,\ldots,i_r\leq M\Big|i_2-i_1,\ldots,i_r-i_1\in S\right\}$$

Now by ignoring the last condition, we have $M$ choices for $i_1$, then $|S|$ choices for $i_2$, $|S|$ choices for $i_3$, and so on, up to $|S|$ choices for $i_r$. Thus, we obtain:
$$K_p^r(a,b)\leq\frac{1}{M^r}\cdot M\cdot|S|\cdot\ldots|S|=\left(\frac{|S|}{M}\right)^{r-1}$$

On the other hand, by Proposition 5.2 our assumption $[(a_y,b_y)]\neq [(a_y,b_{y+1})]$ implies $S\neq \{0,\ldots,M-1\}$. In particular we have $|S|\leq M-1$, and this gives the result.
\end{proof}

With the above estimate in hand, we can now prove:

\begin{theorem}
We have the formula
$$\lim_{r\to\infty}d_p^r(M,N)=\delta_p(M,N)$$
valid for any $p\geq 1$ and any $M,N\in\mathbb N$.
\end{theorem}

\begin{proof}
Our claim is that we have, for any pair of multi-indices $(a,b)$:
$$\lim_{r\to\infty}K_p^r(a,b)=\delta_{[(a_y,b_y)],[(a_y,b_{y+1})]}$$

Indeed, when $[(a_y,b_y)]\neq [(a_y,b_{y+1})]$, this is exactly what we found in Proposition 5.3. As for the remaining case $[(a_y,b_y)]=[(a_y,b_{y+1})]$, this is trivial, because here the equations defining $K_p^r(a,b)$ are all trivial, and so we have $K_p^r(a,b)=1$, for any $r\in\mathbb N$.
\end{proof}

Summarizing, we have proved that the law of the main character for $\mathcal F_{G,H}$ coincides with that computed in \cite{bb2}, for the matrix $\mathcal F_{G\times H}(Q)$, with $Q\in\mathbb T^{G\times H}$ generic. As a consequence, all the findings in \cite{bb2} apply. In what follows we will review these results, by using an analytic approach, and by bringing some technical improvements.

\section{Gram matrices}

We study now the behavior of the limiting moments $\delta_p(M,N)$ that we found, in the $p\to\infty$ limit. For this purpose, let us first recall the following result, from \cite{bb2}:

\begin{proposition}
We have the formula
$$\delta_p(M,N)=\frac{1}{(MN)^p}\int_{\mathbb T^{MN}}Tr(G(Q)^p)dQ$$
where $G\in M_M(C(\mathbb T^{MN}))$ is given by $G(Q)=$ Gram matrix of the rows of $Q$.
\end{proposition}

\begin{proof}
If we denote by $R_1,\ldots,R_M\in\mathbb T^N$ the rows of $Q\in\mathbb T^{MN}$, we have:
\begin{eqnarray*}
\delta_p(M,N)
&=&\frac{1}{(MN)^p}\sum_{a_1\ldots a_p}\sum_{b_1\ldots b_p}\delta_{[a_1b_1,\ldots,a_pb_p],[a_1b_p,\ldots,a_pb_{p-1}]}\\
&=&\frac{1}{(MN)^p}\int_{\mathbb T^{MN}}\sum_{a_1\ldots a_p}\sum_{b_1\ldots b_p}\frac{Q_{a_1b_1}\ldots Q_{a_pb_p}}{Q_{a_1b_p}\ldots Q_{a_pb_{p-1}}}\,dQ\\
&=&\frac{1}{(MN)^p}\int_{\mathbb T^{MN}}\sum_{a_1\ldots a_p}<R_{a_1},R_{a_2}><R_{a_2},R_{a_3}>\ldots<R_{a_p},R_{a_1}>dQ
\end{eqnarray*}

But this gives the formula in the statement, and we are done.
\end{proof}

In the case $M=2$ some simplifications appear, and we have:

\begin{proposition}
We have the formula
$$\delta_p(2,N)=\frac{1}{2^{p-1}}\sum_{k\geq0}\binom{p}{2k}\int_{\mathbb T^N}\left|\frac{q_1+\ldots+q_N}{N}\right|^{2k}dq$$
with the integral at right being with respect to the uniform measure on $\mathbb T^N$.
\end{proposition}

\begin{proof}
We use the formula in Proposition 6.1. If we denote by $R_1,R_2\in\mathbb T^N$ the rows of $Q$ then, with $q=R_1/R_2\in\mathbb T^N$, the Gram matrix that we are interested in is:
$$G(Q)=\begin{pmatrix}N&q_1+\ldots+q_N\\\bar{q}_1+\ldots+\bar{q}_N&N\end{pmatrix}$$

Thus, with $S=(q_1+\ldots+q_N)/N$, we have $G(Q)=NA(q)$, where:
$$A(q)=\begin{pmatrix}1&S\\\bar{S}&1\end{pmatrix}$$

Now since $q\in\mathbb T^N$ is uniform when $Q\in\mathbb T^{2N}$ is uniform, we deduce that we have:
$$\delta_p(2,N)=\frac{1}{2^p}\int_{\mathbb T^N}\sum_{a_1\ldots a_p}A(q)_{a_1a_2}A(q)_{a_2a_3}\ldots A(q)_{a_pa_1}dq$$

The point now is that the nontrivial factors in the above product, namely $S,\bar{S}$, will form together $|S|^k$ factors, with $k\geq0$. To be more precise, in order to find the number of $|S|^{2k}$ summands, we have to count the circular configurations consisting of $p$ numbers $1,2$, such that both the $1$ values and the $2$ values are arranged into $k$ non-empty intervals. By looking at the endpoints of these $2k$ intervals, we have $2\binom{p}{2k}$ choices, so the $k$-th contribution is $C_k=2\binom{k}{2p}|S|^{2k}$. Thus, we have the following formula:
$$\delta_p(2,N)=\frac{1}{2^p}\sum_{k\geq0}2\binom{p}{2k}\int_{\mathbb T^N}|S|^{2k}dq$$

But this gives the formula in the statement, and we are done.
\end{proof}

We write $a_k\simeq b_k$ when $a_k/b_k\to1$. We will need the following result, due to Richmond and Shallit \cite{rsh}:

\begin{proposition}
We have the estimate
$$\int_{\mathbb T^N}\left|\frac{q_1+\ldots+q_N}{N}\right|^{2k}dq\simeq\sqrt{\frac{N^N}{(4\pi k)^{N-1}}}$$
valid in the $k\to\infty$ limit.
\end{proposition}

\begin{proof}
This is a reformulation of the result in \cite{rsh}. Observe first that we have:
\begin{eqnarray*}
\int_{\mathbb T^N}\Big|q_1+\ldots+q_N\Big|^{2k}dq
&=&\int_{\mathbb T^N}\sum_{i_1\ldots i_k}\sum_{j_1\ldots j_k}\frac{q_{i_1}\ldots q_{i_k}}{q_{j_1}\ldots q_{j_k}}dq\\
&=&\#\left\{\begin{matrix}i_1\ldots i_k\in\{0,\ldots,N-1\}\\ j_1\ldots j_k\in\{0,\ldots,N-1\}\end{matrix}\Big|\begin{matrix}\ [i_1,\ldots,i_k]\\ =[j_1,\ldots,j_k]\end{matrix}\right\}
\end{eqnarray*}

Let us examine now the numbers on the right. If we denote by $r_1,\ldots,r_N$ the number of occurrences of $0,\ldots,N-1$ in the set with repetitions $[i]=[j]$, then $r_1+\ldots+r_N=k$, and the corresponding solutions of $[i]=[j]$ come by dividing, once for $i$, and once for $j$, the set $\{1,\ldots,k\}$ into subsets of size $r_1,\ldots,r_N$. Thus, we have:
$$\int_{\mathbb T^N}\Big|q_1+\ldots+q_N\Big|^{2k}dq=\sum_{k=\Sigma r_i}\binom{k}{r_1,\ldots,r_N}^2$$

By using now the estimate in \cite{rsh}, we obtain the result.
\end{proof}

We can now deduce a final estimate at $M=2$, as follows:

\begin{theorem}
We have the estimate
$$\delta_p(2,N)\simeq\sqrt{\frac{N^N}{(\pi p)^{N-1}}}$$
valid in the $p\to\infty$ limit.
\end{theorem}

\begin{proof}
We use the formula in Proposition 6.2. Since for any $T>0$ the values $k<T$ won't contribute to the $p\to\infty$ limit, we can use Proposition 6.3, and we obtain:
$$\delta_p(2,N)\simeq\sqrt{\frac{N^N}{(2\pi)^{N-1}}}\cdot\frac{1}{2^{p-1}}\sum_{k\geq0}\binom{p}{2k}\frac{1}{\sqrt{(2k)^{N-1}}}$$

Let us denote by $A_{even}$ the average of $2^{p-1}$ terms on the right. This average is indexed by the integers $s=2k$ in an obvious way, and we can consider as well the ``complementary'' quantity $A_{odd}$, indexed by the integers $s=2k+1$. By estimating $|A_{even}-A_{odd}|$ we deduce that we have $A_{even}\simeq A_{odd}$, and so $A_{even}\simeq\frac{A_{even}+A_{odd}}{2}$. Thus, we have:
$$\delta_p(2,N)\simeq\sqrt{\frac{N^N}{(2\pi)^{N-1}}}\cdot\frac{1}{2^p}\sum_{s\geq0}\binom{p}{s}\frac{1}{\sqrt{s^{N-1}}}$$

On the other hand, by derivating several times the binomial formula $(1+x)^p=\sum_{s\geq0}\binom{p}{s}x^s$, and then evaluating at $x=1$, we have the following estimate:
$$\frac{1}{2^p}\sum_{s\geq0}\binom{p}{s}s^\alpha\simeq\left(\frac{p}{2}\right)^\alpha$$

With $\alpha=(1-N)/2$, this gives the following formula:
$$\delta_p(2,N)\simeq\sqrt{\frac{N^N}{(2\pi)^{N-1}}}\cdot\sqrt{\left(\frac{2}{p}\right)^{N-1}}$$

But this gives the formula in the statement, and we are done.
\end{proof}

\section{Partition decomposition} 

Our purpose now will be that of estimating $\delta_p(M,N)$, when $M,N\in\mathbb N$ are arbitrary. The idea will be that of decomposing over partitions. First, we have:

\begin{proposition}
We have the formula
$$\delta_p(M,N)=\frac{1}{(MN)^p}\sum_{\pi\triangleright\sigma}\frac{M!}{(M-|\pi|)!}\cdot\frac{N!}{(N-|\sigma|)!}$$
where for $\pi,\sigma\in P(p)$ we write $\pi\triangleright\sigma$ when $|\beta\cap\gamma|=|(\beta-1)\cap\gamma|,\forall\beta\in\pi,\forall\gamma\in\sigma$.
\end{proposition}

\begin{proof}
We know that $\delta_p(M,N)$ is the probability for $[(a_x,b_x)]=[(a_x,b_{x+1})]$ to happen. We can split this quantity over pairs of partitions, as follows:
$$\delta_p(M,N)=\frac{1}{(MN)^p}\sum_{\pi,\sigma\in P(p)}\#\left\{\begin{matrix}a_1,\ldots,a_p\in\{0,\ldots,M-1\}\\ b_1,\ldots,b_p\in\{0,\ldots,N-1\}\end{matrix}\Big|\begin{matrix} \ker a=\pi,\ \ \ker b=\sigma\\ [(a_x,b_x)]=[(a_x,b_{x+1})]\end{matrix}\right\}$$

Now observe that the validity of the condition $[(a_x,b_x)]=[(a_x,b_{x+1})]$ depends only on the partitions $\pi=\ker a,\sigma=\ker b$. To be more precise, this condition is satisfied precisely when the condition $\pi\triangleright\sigma$ in the statement holds. We therefore obtain:
$$\delta_p(M,N)=\frac{1}{(MN)^p}\sum_{\pi\triangleright\sigma}\#\left\{\begin{matrix}a_1,\ldots,a_p\in\{0,\ldots,M-1\}\\ b_1,\ldots,b_p\in\{0,\ldots,N-1\}\end{matrix}\Big|\begin{matrix}\ \ker a=\pi\\ \ \ker b=\sigma\end{matrix}\right\}$$

But this gives the formula in the statement, and we are done.
\end{proof}

As an application, we can discuss what happens in the $M=tN\to\infty$ regime, which means $N\to\infty$ and $M=tN+o(1)$, with $t>0$ fixed. The result, from \cite{bb2}, is:

\begin{proposition}
With $M=tN\to\infty$ we have
$$\delta_p(M,N)\simeq S_p(t)M^{-p}N$$
where $S_p(t)=\sum_{\pi\in NC(p)}t^{|\pi|}$ is the Stirling polynomial of $NC(p)$.
\end{proposition}

\begin{proof}
According to the formula in Proposition 7.1, with $M=tN\to\infty$ we have:
$$\delta_p(M,N)\simeq\sum_{\pi\triangleright\sigma}M^{|\pi|-p}N^{|\sigma|-p}$$

We use now the standard fact that $\pi\triangleright\sigma$ implies $|\pi|+|\sigma|\leq p+1$, with equality when $\pi,\sigma\in NC(p)$ are inverse to each other, via Kreweras complementation. We obtain:
$$\delta_p(M,N)\simeq\sum_{\pi\in NC(p)}M^{|\pi|-p}N^{1-|\pi|}$$

But this gives the formula in the statement, and we are done. See \cite{bb2}.
\end{proof}

Now back to our original question, concerning the case where $M,N\in\mathbb N$ are fixed, we can rewrite the formula in Proposition 7.1 in a more convenient way, as follows:

\begin{proposition}
We have the formula
$$\delta_p(M,N)=\sum_{s=1}^M\sum_{t=1}^N\frac{M!}{(M-s)!}\cdot\frac{S_{ps}}{M^p}\cdot\frac{N!}{(N-t)!}\cdot\frac{S_{pt}}{N^p}\cdot P\left(\pi\triangleright\sigma\Big||\pi|=s,|\sigma|=t\right)$$
where $S_{ps}=\#\{\pi\in P(p)||\pi|=s\}$ are the Stirling numbers of $P(p)$.
\end{proposition}

\begin{proof}
According to the formula in Proposition 7.1, we have:
$$\delta_p(M,N)=\frac{1}{(MN)^p}\sum_{s=1}^M\sum_{t=1}^N\frac{M!}{(M-s)!}\cdot\frac{N!}{(N-t)!}\#\left(\pi\triangleright\sigma\Big||\pi|=s,|\sigma|=t\right)$$

On the other hand, the probability in the statement is given by:
$$P\left(\pi\triangleright\sigma\Big||\pi|=s,|\sigma|=t\right)=\frac{\#\left(\pi\triangleright\sigma\Big||\pi|=s,|\sigma|=t\right)}{S_{ps}S_{pt}}$$

By combining these two formulae, we obtain the result.
\end{proof}

Consider the probabilities which appear on the right in Proposition 7.3:
$$\varepsilon_p(s,t)=P\left(\pi\triangleright\sigma\Big||\pi|=s,|\sigma|=t\right)$$

The corresponding contributions to $\delta_p(M,N)$ are then given by:
$$\delta_p^{st}(M,N)=\frac{M!}{(M-s)!}\cdot\frac{S_{ps}}{M^p}\cdot\frac{N!}{(N-t)!}\cdot\frac{S_{pt}}{N^p}\cdot\varepsilon_p(s,t)$$

The idea now will be to separate the contributions coming from indices $s=1$ or $t=1$. To be more precise, we can rewrite Proposition 7.3 as follows:

\begin{theorem}
We have the formula
$$\delta_p(M,N)=\frac{1}{M^{p-1}}+\frac{1}{N^{p-1}}-\frac{1}{(MN)^{p-1}}+\sum_{s=2}^M\sum_{t=2}^N\delta_p^{st}(M,N)$$
where $\delta_p^{st}(M,N)$ are the contributions defined above.
\end{theorem}

\begin{proof}
According to Proposition 7.3, we have the following formula:
$$\delta_p(M,N)=\sum_{s=1}^M\sum_{t=1}^N\delta_p^{st}(M,N)$$

Since we have $\varepsilon_p(1,t)=1$, the contributions at $s=1$ are given by:
$$\delta_p^{1t}(M,N)=M\cdot\frac{1}{M^p}\cdot\frac{N!}{(N-t)!}\cdot\frac{S_{pt}}{N^p}=\frac{1}{M^{p-1}}\cdot\frac{N!}{(N-t)!}\cdot\frac{S_{pt}}{N^p}$$

Now by summing over $t\geq1$, we obtain the following formula:
$$\sum_{t=1}^N\delta_p^{1t}(M,N)=\frac{1}{M^{p-1}}\sum_{t=1}^N\frac{N!}{(N-t)!}\cdot\frac{S_{pt}}{N^p}=\frac{1}{M^{p-1}}$$

Similarly, we have as well the following formula:
$$\sum_{s=1}^M\delta_p^{s1}(M,N)=\frac{1}{N^{p-1}}\sum_{s=1}^M\frac{M!}{(M-s)!}\cdot\frac{S_{ps}}{M^p}=\frac{1}{N^{p-1}}$$

Finally, at $s=1,t=1$ the contribution is as follows:
$$\delta_p^{11}(M,N)=M\cdot\frac{1}{M^p}\cdot N\cdot\frac{1}{N^p}=\frac{1}{(MN)^{p-1}}$$

By using the inclusion-exclusion principle, this gives the result.
\end{proof}

\section{Moment estimates}

In this section we estimate $\delta_p(M,N)$, by using the formula found in Theorem 7.4. In order to deal with the contributions at $s\geq2,t\geq2$, we use the following fact:

\begin{proposition}
The function constructed above,
$$\varepsilon_p(s,t)=P\left(\pi\triangleright\sigma\Big||\pi|=s,|\sigma|=t\right)$$
is decreasing in both $s\in\mathbb N$ and $t\in\mathbb N$.
\end{proposition}

\begin{proof}
The problem being symmetric in $s,t$, it is enough to prove that $\varepsilon_p(s,t)$ is decreasing in $t$. By splitting the problem over the partitions $\pi$ satisfying $|\pi|=s$, it is enough to prove that for any partition $\pi\in P(p)$, the following quantity is decreasing with $t$:
$$\varepsilon_\pi(t)=P\left(\pi\triangleright\sigma\Big||\sigma|=t\right)$$

In order to do so, recall from Proposition 7.1 that $\pi\triangleright\sigma$ is equivalent to:
$$|\beta\cap\gamma|=|(\beta-1)\cap\gamma|,\forall\beta\in\pi,\forall\gamma\in\sigma$$

Now observe that when merging two blocks of $\sigma$, say $(\gamma_1,\gamma_2)\to\gamma$, the condition is satisfied for $\gamma$, simply by summing the equalities for $\gamma_1,\gamma_2$. We deduce from this that the probability $\varepsilon_\pi(t)$ gets bigger when decreasing the number $t=|\sigma|$, as desired.
\end{proof}

Let us combine now Theorem 7.4 with Proposition 8.1. We obtain:

\begin{proposition}
We have the estimate
$$\delta_p(M,N)\leq1-\left(1-\frac{1}{M^{p-1}}\right)\left(1-\frac{1}{N^{p-1}}\right)\big(1-\varepsilon_p(2,2)\big)$$
valid for any $M,N\geq2$.
\end{proposition}

\begin{proof}
The formula in Theorem 7.4 above can be written as follows:
\begin{eqnarray*}
\delta_p(M,N)
&=&\frac{1}{M^{p-1}}+\frac{1}{N^{p-1}}-\frac{1}{(MN)^{p-1}}+\sum_{s=2}^M\sum_{t=2}^N\delta_p^{st}(M,N)\\
&=&1-\left(1-\frac{1}{M^{p-1}}\right)\left(1-\frac{1}{N^{p-1}}\right)+\sum_{s=2}^M\sum_{t=2}^N\delta_p^{st}(M,N)
\end{eqnarray*}

According now to Proposition 8.1, for any $s,t\geq2$ we have:
$$\delta_p^{st}(M,N)\leq\frac{M!}{(M-s)!}\cdot\frac{S_{ps}}{M^p}\cdot\frac{N!}{(N-t)!}\cdot\frac{S_{pt}}{N^p}\cdot\varepsilon_p(2,2)$$

Now by summing over all indices $s,t\geq2$, and by using the inclusion-exclusion principle, as in the proof of Theorem 7.4, we obtain:
\begin{eqnarray*}
\sum_{s=2}^M\sum_{t=2}^N\delta_p^{st}(M,N)
&\leq&\sum_{s=2}^M\sum_{t=2}^N\frac{M!}{(M-s)!}\cdot\frac{S_{ps}}{M^p}\cdot\frac{N!}{(N-t)!}\cdot\frac{S_{pt}}{N^p}\cdot\varepsilon_p(2,2)\\
&=&\left(1-\frac{1}{M^{p-1}}-\frac{1}{N^{p-1}}+\frac{1}{(MN)^{p-1}}\right)\varepsilon_p(2,2)\\
&=&\left(1-\frac{1}{M^{p-1}}\right)\left(1-\frac{1}{N^{p-1}}\right)\varepsilon_p(2,2)
\end{eqnarray*}

But this gives the formula in the statement, and we are done.
\end{proof}

On the other hand, by using the results obtained in section 6 above, we have:

\begin{proposition}
We have the estimate
$$\varepsilon_p(2,N)\simeq\frac{1}{2\cdot N!}\sqrt{\frac{N^N}{(\pi p)^{N-1}}}$$
valid in the $p\to\infty$ limit.
\end{proposition}

\begin{proof}
We have the following estimate, in the $p\to\infty$ limit:
\begin{eqnarray*}
\delta_p^{st}(M,N)
&=&\frac{M!}{(M-s)!}\cdot\frac{S_{ps}}{M^p}\cdot\frac{N!}{(N-t)!}\cdot\frac{S_{pt}}{N^p}\cdot\varepsilon_p(s,t)\\
&\simeq&\frac{M!}{(M-s)!}\cdot\frac{s^p}{M^p}\cdot\frac{N!}{(N-t)!}\cdot\frac{t^p}{N^p}\cdot\varepsilon_p(s,t)\\
&=&\frac{M!}{(M-s)!}\cdot\frac{N!}{(N-t)!}\left(\frac{st}{MN}\right)^p\varepsilon_p(s,t)
\end{eqnarray*}

Here we have used the estimate $S_{ps}\simeq s^p$, which follows from the fact that choosing a partition $\pi\in P(p)$ with $\leq s$ blocks amounts in assigning a number $1,\ldots,s$ to any of the points $1,\ldots,p$, and the assignements which lead to $|\pi|<s$ can be neglected.

In particular, at $s=M=2$ we obtain:
$$\delta_p^{2t}(2,N)\simeq 2\cdot\frac{N!}{(N-t)!}\left(\frac{t}{N}\right)^p\varepsilon_p(2,t)$$

By combining this estimate with Theorem 7.4 at $M=2$, we obtain:
\begin{eqnarray*}
\delta_p(2,N)
&=&\frac{1}{2^{p-1}}+\frac{1}{N^{p-1}}-\frac{1}{(2N)^{p-1}}+\sum_{t=2}^N\delta_p^{2t}(2,N)\\
&\simeq&\frac{1}{2^{p-1}}+2\sum_{t=2}^N\frac{N!}{(N-t)!}\left(\frac{t}{N}\right)^p\varepsilon_p(2,t)
\end{eqnarray*}

With this formula in hand, we can proceed by recurrence on $N\geq2$. Since the quantity in the statement converges with $p\to\infty$ to $0$ much slower than the various powers $\alpha^N$, with $\alpha\in(0,1)$, only the last term will matter, and our estimate simply reads:
$$\delta_p(2,N)\simeq 2\cdot N!\varepsilon_p(2,N)$$

Now by using the $M=2$ estimate from Theorem 6.4, we obtain:
$$\varepsilon_p(2,N)\simeq\frac{1}{2\cdot N!}\cdot\delta_p(2,N)\simeq\frac{1}{2\cdot N!}\sqrt{\frac{N^N}{(\pi p)^{N-1}}}$$

Thus we have obtained the formula in the statement.
\end{proof}

With the above results in hand, we can now prove our result:

\begin{theorem}
We have $\lim_{p\to\infty}\delta_p(M,N)=0$, for any $M,N\geq2$.
\end{theorem}

\begin{proof}
By combining Proposition 8.2 and Proposition 8.3, we obtain:
$$\delta_p(M,N)\leq1-\left(1-\frac{1}{M^{p-1}}\right)\left(1-\frac{1}{N^{p-1}}\right)\big(1-\varepsilon_p(2,2)\big)$$

Since the product on the right converges to $1\times 1\times 1=1$, this gives the result.
\end{proof}

\section{Poisson laws}

We recall that the free analogue of the Poisson law of parameter $t>0$, in the sense of the Bercovici-Pata bijection \cite{bpa}, is the Marchenko-Pastur law of parameter $t$, also called free Poisson law of parameter $t$. We denote this measure by $\pi_t$. See \cite{mpa}, \cite{nsp}, \cite{vdn}.

We have the following result, summarizing our findings:

\begin{theorem}
Given two finite abelian groups $G,H$, with $|G|=M,|H|=N$, consider the main character $\chi$ of the quantum group associated to $\mathcal F_{G\times H}$.
\begin{enumerate}
\item $\mu=law(\frac{\chi}{MN})$ is supported on $[0,1]$.

\item This measure $\mu$ has no atom at $1$.

\item With $M=tN\to\infty$ we have $law\left(\frac{\chi}{N}\right)=\left(1-\frac{1}{M}\right)\delta_0+\frac{1}{M}\,\pi_t$, in moments.
\end{enumerate}
\end{theorem}

\begin{proof}
In this statement (1) is trivial, (2) is new, and (3) is since known since \cite{bb2}, in the case of the generic fibers. To be more precise, the proof goes as follows:

(1) This follows from the fact that $\chi$ is by definition the main character for a certain quantum group $\mathcal G\subset S_{MN}^+$, and is therefore a sum of $MN$ projections.

(2) This follows from Theorem 8.4 above, and from the fact that an atom at $1$ would make the moments converge to a nonzero quantity.

(3) According to our various normalizations, we have:
$$\int_\mathcal G^r\left(\frac{\chi}{N}\right)^p=\frac{c_p^r(M,N)}{N^p}
=\frac{(MN)^{p-1}d_p^r(M,N)}{N^p}=\frac{M^{p-1}}{N}d_p^r(M,N)$$

By using Proposition 7.2 we obtain, in the $M=tN\to\infty$ limit:
$$\int_\mathcal G\left(\frac{\chi}{N}\right)^p\simeq\frac{M^{p-1}}{N}\delta_p(M,N)\simeq\frac{M^{p-1}}{N}S_p(t)M^{-p}N=\frac{1}{M}S_p(t)$$

Now since $S_p(t)$ is the $p$-th moment of $\pi_t$, this gives the result.
\end{proof}

\section{Concluding remarks}

There are several questions, in relation with the above results. First, we do not know how to improve Theorem 8.4, with a precise estimate, as in Theorem 6.4. 

There are as well some interesting questions in relation with \cite{ba1}, \cite{tzy}. The main problem here, well-known and open, is that of understanding how a general deformed Fourier matrix $\mathcal F_K$ can be defined, directly in terms of the finite abelian group $K$. 

In relation now with \cite{bne}, observe that the representations there are as well of the form $\pi:C(S_{\dim B}^+)\to C(U_B,\mathcal L(B))$, for a certain finite dimensional $C^*$-algebra $B$. In the present paper this algebra is a commutative one, $B=C(G\times H)$. We believe that the unification with \cite{bne} is an important question, which could lead to a substantial ``boost'' in the understanding and use of the integration formula in \cite{bfs}, \cite{wa2}.


\begin{thebibliography}{99}

\bibitem{ba1}T. Banica, First order deformations of the Fourier matrix, {\em J. Math. Phys.} {\bf 55} (2014), 1--22.

\bibitem{ba2}T. Banica, Truncation and duality results for Hopf image algebras, {\em Bull. Pol. Acad. Sci. Math.} {\bf 62} (2014), 161--179.

\bibitem{bb1}T. Banica and J. Bichon, Hopf images and inner faithful representations, {\em Glasg. Math. J.} {\bf 52} (2010), 677--703.

\bibitem{bb2}T. Banica and J. Bichon, Random walk questions for linear quantum groups, {\em Int. Math. Res. Not.} {\bf 24} (2015), 13406--13436.

\bibitem{bco}T. Banica and B. Collins, Integration over compact quantum groups, {\em Publ. Res. Inst. Math. Sci.} {\bf 43} (2007), 277--302.

\bibitem{bfs}T. Banica, U. Franz and A. Skalski, Idempotent states and the inner linearity property, {\em Bull. Pol. Acad. Sci. Math.} {\bf 60} (2012), 123--132.

\bibitem{bne}T. Banica and I. Nechita, Flat matrix models for quantum permutation groups, {\em Adv. Appl. Math.} {\bf 83} (2017), 24--46.

\bibitem{bsp}T. Banica and R. Speicher, Liberation of orthogonal Lie groups, {\em Adv. Math.} {\bf 222} (2009), 1461--1501.

\bibitem{bpa}H. Bercovici and V. Pata, Stable laws and domains of attraction in free probability theory, {\em Ann. of Math.} {\bf 149} (1999), 1023--1060.

\bibitem{bic}J. Bichon, Quotients and Hopf images of a smash coproduct, {\em Tsukuba J. Math.} {\bf 39} (2015), 285--310.

\bibitem{bcv}M. Brannan, B. Collins and R. Vergnioux, The Connes embedding property for quantum group von Neumann algebras, preprint 2014.

\bibitem{chi}A. Chirvasitu, Residually finite quantum group algebras, {\em J. Funct. Anal.} {\bf 268} (2015), 3508--3533.

\bibitem{csn}B. Collins and P. \'Sniady, Integration with respect to the Haar measure on the unitary, orthogonal and symplectic group, {\em Comm. Math. Phys.} {\bf 264} (2006), 773--795.

\bibitem{dit}P. Di\c t\u a, Some results on the parametrization of complex Hadamard matrices, {\em J. Phys. A} {\bf 37} (2004), 5355--5374.

\bibitem{fsk}U. Franz and A. Skalski, On idempotent states on quantum groups, {\em J. Algebra} {\bf 322} (2009), 1774--1802.

\bibitem{mpa}V.A. Marchenko and L.A. Pastur, Distribution of eigenvalues in certain sets of random matrices, {\em Mat. Sb.} {\bf 72} (1967), 507--536.

\bibitem{nsp}A. Nica and R. Speicher, Lectures on the combinatorics of free probability, Cambridge Univ. Press (2006).

\bibitem{rwe}S. Raum and M. Weber, The full classification of orthogonal easy quantum groups, {\em Comm. Math. Phys.} {\bf 341} (2016), 751--779.

\bibitem{rsh}L.B. Richmond and J. Shallit, Counting abelian squares, {\em Electron. J. Combin.} {\bf 16} (2009), 1--9.

\bibitem{sso}A. Skalski and P. So\l tan, Quantum families of invertible maps and related problems, {\em Canad. J. Math.} {\bf 68} (2016), 698--720.

\bibitem{tzy}W. Tadej and K. \.Zyczkowski, Defect of a unitary matrix, {\em Linear Algebra Appl.} {\bf 429} (2008), 447--481.

\bibitem{vdn}D.V. Voiculescu, K.J. Dykema and A. Nica, Free random variables, AMS (1992).

\bibitem{wa1}S. Wang, Quantum symmetry groups of finite spaces, {\em Comm. Math. Phys.} {\bf 195} (1998), 195--211.

\bibitem{wa2}S. Wang, $L_p$-improving convolution operators on finite quantum groups, {\em Indiana Univ. Math. J.} {\bf 65} (2016), 1609--1637.

\bibitem{wo1}S.L. Woronowicz, Compact matrix pseudogroups, {\em Comm. Math. Phys.} {\bf 111} (1987), 613--665.

\bibitem{wo2}S.L. Woronowicz, Tannaka-Krein duality for compact matrix pseudogroups. Twisted SU(N) groups, {\em Invent. Math.} {\bf 93} (1988), 35--76.

\end{thebibliography}
\end{document}